\documentclass[11pt,amssymb]{amsart}

\usepackage[OT2,T1]{fontenc}
\DeclareSymbolFont{cyrletters}{OT2}{wncyr}{m}{n}
\newcommand{\Sha}{\text{{\brus
SH}}}
\usepackage[mathscr]{eucal}

\usepackage{graphicx}
\usepackage{amssymb}
\usepackage{amsmath}
\usepackage{color}
\usepackage{amsmath,amscd}
\usepackage{pictexwd,dcpic}
\newtheorem{theorem}{Theorem}
\newtheorem{lemma}[theorem]{Lemma}
\newtheorem{proposition}[theorem]{Proposition}

\newtheorem{corollary}[theorem]{Corollary}

\newcommand{\Gal}{\mbox{Gal}}

\newcommand\bQ{\mathbb Q}

\newcommand\bZ{\mathbb Z}
\newcommand\cL{\mathcal L}

\newcommand\cD{\mathcal D}

\newcommand\cP{\mathcal P}
\newcommand\cR{\mathcal R}
\newcommand\cS{\mathcal S}
\newcommand\cT{\mathcal T}

\font\brus=wncyr10.240pk scaled 1200 .240pk
\begin{document}
\title{On The Multinorm Principle For Finite Extensions}
\author[Pollio]{Timothy P. Pollio}

\begin{abstract} Let $L_1$ and $L_2$ be finite abelian extensions of a global field $K$.  We compute the obstruction to the multinorm principle for the pair $L_1,L_2$.  
\end{abstract}

\maketitle

\section{Introduction}
Let $K$ be a global field.  Given a finite extension $L/K$, let $J_L$ denote the idele group of $L$, let $N_{L/K}\colon J_L \rightarrow J_K$ denote the natural extension of the norm map associated with $L/K$, and let $\Sha(L/K)$ denote the Tate-Shafarevich group of $L/K$ (cf. \cite{PoR}).  When $K$ does not vary, we will write $N(L^\times)$ and $N(J_L)$ respectively in place of $N_{L/K}(L^\times)$ and $N_{L/K}(J_L)$.  As in \cite{PoR}, we say that a pair of finite extensions $L_1,L_2$ of $K$ satisfies the \textit{multinorm principle} if 
\[ K^\times \cap N(J_{L_1})N(J_{L_2})=N(L_1^\times)N(L_2^\times).\]

The obstruction to the multinorm principle is given by the quotient
\[\Sha(L_1,L_2/K):= \frac{K^\times \cap N(J_{L_1})N(J_{L_2})}{N(L_1^\times)N(L_2^\times)}.  \]
The multinorm principle has a variety of applications (cf. \textit{loc. cit.}, $\S1$), but it is not fully understood.  The main theorem of \cite{PoR} says that $\Sha(L_1,L_2/K)=\{1\}$ whenever $L_1,L_2$ is a pair of finite separable extensions of $K$ with linearly disjoint Galois closures.  However, little is known about $\Sha(L_1,L_2/K)$ for more general pairs of extensions.  

In this paper, we describe a general approach to the multinorm principle that builds on the techniques used in \cite{PoR}.  The idea is that we should try to describe $\Sha(L_1,L_2/K)$ by studying the map
\[f \colon \Sha(L_1/K) \times \Sha(L_2/K) \rightarrow \Sha(L_1,L_2/K) \]
defined by
\[\left(x  N(L_1^\times) ,yN(L_2^\times)    \right) \mapsto xy^{-1}N(L_1^\times)N(L_2^\times).\]
\vskip 3mm
Since $\Sha(L_1,L_2/K)$ is determined up to extension by $\mathrm{Im}\mbox{ }f$ and $\mathrm{Coker}\mbox{ }f$, it suffices to describe     
\[ \Sha_1(L_1,L_2/K):= \mathrm{Im}(f) \mbox{ }\mbox{ }\mbox{ }\mbox{ and} \mbox{ }\mbox{ }\mbox{ }\Sha_2(L_1,L_2/K):= \mathrm{Coker}(f)\]
individually, and we refer to these groups as the \textit{first and second obstructions} to the multinorm principle.  We will analyze $\Sha_1(L_1,L_2/K)$ and $\Sha_2(L_1,L_2/K)$ using group cohomology and class field theory respectively, and prove estimates which will allow us to compute both obstructions in some important special cases.  This approach can be used to recover the main theorem of \cite{PoR} (see \S 5), and it enables us to prove the main theorem of this note, which characterizes the multinorm principle for pairs of finite abelian extensions.        \begin{theorem}Let $L_1,L_2$ be a pair of finite abelian extensions of $K$.  Then
\[\Sha(L_1,L_2/K) \simeq \Sha(L_1 \cap L_2/K) .\]
In particular, $L_1 , L_2$ satisfies the multinorm principle iff $L_1 \cap L_2$ satisfies the norm principle.   
\end{theorem}     

We begin in $\S2$ and $\S3$ by giving descriptions of the first and second obstructions respectively.  Then we apply these descriptions in the special case where $L_1, L_2$ is a pair of abelian extensions of $K$ to prove Theorem 1 in $\S4$.  Finally, $\S5$ contains several examples related to our analysis of the multinorm principle.

\vskip 3mm
\noindent \textbf{Remark.} Earlier versions of this paper contained additional sufficient conditions for the multinorm principle to hold for $n$-tuples of extensions (see https://sites.google.com/site/timothypollio/papers).  By generalizing the proof of Proposition $15$ in \cite{PoR}, we were able to show that the multinorm principle holds for $n$-tuples of finite separable extensions whose Galois closures are linearly disjoint as a family and for pairs of Galois extensions with cyclic intersection.  Meanwhile, Demarche and Wei obtained similar results which they describe in \cite{DWei}.  Their argument is similar to ours and their results are slightly stronger, so we decided to omit these results from the final version of this paper.

\section{The First Obstruction}
Our analysis of the first obstruction begins with the following observation.
\begin{lemma}\label{FirstRed} Let $L_1, L_2$ be a pair of finite extensions of $K$, let $L=L_1L_2$, and let   
\[ g \colon \Sha(L/K) \rightarrow \Sha(L_1/K) \times \Sha(L_2/K)  \]     
be the map defined by
\[x  N(L^\times)    \mapsto \left(xN(L_1^\times),xN(L_2^\times)   \right).\]
Then $\Sha_1(L_1,L_2/K)$ is isomorphic to a quotient of $\mathrm{Coker}(g)$.
\end{lemma}
\begin{proof} It is clear that $\mathrm{Im}(g) \subset \mathrm{Ker}(f)$, so the claim follows from the first isomorphism theorem.    
\end{proof}

In this section, we give several results which can be used to compute $\mathrm{Coker}(g)$ when $L_1$ and $L_2$ are Galois extensions of $K$.  More generally, if $F$ and $L$ are Galois extensions of $K$ with $L \subset F$, we will describe the map   
\[h\colon  \Sha(F/K)= \frac{K^\times \cap N(J_F)}{N(F^\times)} \rightarrow \frac{K^\times \cap N(J_L)}{N(L^\times)} = \Sha(L/K) \]
defined by 
\[xN(F^\times) \mapsto xN(L^\times). \]
Given a finite group $G$ and a $G$-module $A$, we let $\hat{H}^i(G,A)$ denote the $i$th Tate cohomology group of $A$, and let $\mathrm{Cor}$, $\mathrm{Def}$, and $\mathrm{Rsd}$ denote the corestriction, deflation(cf. \cite{Weiss}), and residuation (cf. \cite{Horie}) maps.  We consider $\bZ$ as a $G$-module with trivial action.  Let $H_i(G,A)$ denote the $i$th homology group of $A$.  If $\varphi \colon G \rightarrow H$ is a group homomorphism, then $\varphi$ induces a map of standard complexes which induces a map of homology groups
\[\varphi_* \colon H_i(G,\bZ) \rightarrow H_i(H,\bZ) \] 
(cf. \cite[p.\,99]{Cass}).  When $A=\bZ$, the corestriction and residuation maps can both be interpreted as induced maps in this sense.
\begin{proposition}\label{generalind} Let $i \geq 1$ and identify $H_i(-,\bZ)$ with $\hat{H}^{-i-1}(-,\bZ).$  If $H \leq G$ and $\iota \colon H \rightarrow G$ is the canonical inclusion map, then 
\[\mathrm{Cor}^G_H \colon H_i(H,\bZ) \rightarrow H_i(G,\bZ) \]
is equal to $\iota_*$.  Furthermore, if $H \triangleleft G$ and $\pi \colon G \rightarrow G/H$ is the canoncial projection map, then 
\[\mathrm{Rsd}^G_{G/H}\colon H_i(G,\bZ) \rightarrow H_i(G/H,\bZ) \]
is equal to $\pi_*$.
\end{proposition}
\begin{proof} This follows from the definitions of the corestriction and residuation maps given in \cite[p.\,99]{Cass} and \cite{Horie}.
\end{proof}
Let $G= \Gal(F/K)$ and $H=\Gal(L/K)$.  For each valuation $v$ of $K$, let $H^v$ and $G^v$ be the decomposition groups of compatible fixed extensions of $v$ to $L$ and $F$ respectively.  Let $\iota^v_G\colon G^v \rightarrow G$ and $\iota^v_H\colon H^v \rightarrow H$ denote the canoncial inclusion maps, and let $\pi\colon G \rightarrow H$ and $\pi^v\colon G^v \rightarrow H^v$ denote the canonical projection maps.  We note that $\iota^v_H \circ \pi^v = \pi \circ \iota^v_G$, and that the construction of induced maps is functorial so the following is a consequence of Proposition \ref{generalind}.
\begin{corollary}\label{compat2} For $i \geq 1$ and $A=\bZ$, 
\[ \mathrm{Cor}_{G^v}^G = (\iota_G^v)_*,  \mbox{ }\mbox{ }\mbox{ } \mathrm{Cor}_{H^v}^H = (\iota_H^v)_*, \mbox{ }\mbox{ }\mbox{ } \mathrm{Rsd}^G_{H} = \pi_*\mbox{ }\mbox{ }\mbox{ }\mbox{ and }\mbox{ }\mbox{ }\mbox{ }\mbox{ } \mathrm{Rsd}^{G^v}_{H^v} = \pi^v_*. \]
Moreover, the diagram
\[ 
\begin{CD}
H_2(G^v,\bZ) @>\mathrm{Cor}^G_{G^v}>> H_2(G,\bZ) \\
@V\mathrm{Rsd}^{G^v}_{H^v}VV @V\mathrm{Rsd}^G_{H}VV \\
H_2(H^v,\bZ) @>\mathrm{Cor}^{H}_{H^v}>> H_2(H,\bZ)
\end{CD}
\]
commutes.  
\end{corollary}
\begin{lemma}\label{ShaLem}

The diagram
\begin{equation}\label{H2Diag1}
\begin{CD}
\bigoplus_v H_2(G^v,\bZ) @> \sum_v \mathrm{Cor}_{G^v}^G>> H_2(G,\bZ) \\
@V \left( \mathrm{Rsd}_{H^v}^{G^v} \right) VV @V\mathrm{Rsd}^G_{H}VV \\
\bigoplus_v H_2(H^v,\bZ) @> \sum_v \mathrm{Cor}_{H^v}^H>> H_2(H,\bZ)
\end{CD}\end{equation}
commutes.  Let
\[\gamma_F= \sum_v \mathrm{Cor}_{G^v}^G \mbox{ }\mbox{ }\mbox{ }\mbox{ and  }\mbox{ }\mbox{ }\mbox{ }
\gamma_L= \sum_v \mathrm{Cor}_{H^v}^H.\]
 Then $h$ can be identified with the map
\[\cR_{F/L}\colon \mathrm{Coker}(\gamma_F)\rightarrow \mathrm{Coker}(\gamma_L) \]  
induced by $\mathrm{Rsd}^G_{H}$.
\end{lemma}  
\begin{corollary}\label{sha1cor} $g$ can be identified with the map
\[\cR_{L/L_1} \times \cR_{L/L_2} \colon \mathrm{Coker}(\gamma_{L}) \rightarrow \mathrm{Coker}(\gamma_{L_1}) \times \mathrm{Coker}(\gamma_{L_2}). \]
\end{corollary}
\noindent {\it Proof of Lemma \ref{ShaLem}.} 
Both parts of the lemma are proved in \cite[\S4]{Horie}.  For completeness, we reproduce the argument here.  The reader may find it helpful to consult the properties of the deflation and residuation maps described in \cite[\S3]{PoR} and to compare the following with the proof of proposition 5 of \cite{PoR}.  

Let $C_L$ and $C_F$ denote the idele class groups of $L$ and $F$ respectively, and identify $H$ with the quotient $G / \mathrm{Gal}(F/L)$.  Since the $\mathrm{Gal}(F/L)$-fixed points of 
\begin{equation}\label{seseq1}
1 \rightarrow F^\times \rightarrow J_F \rightarrow C_F \rightarrow 1
\end{equation}
 form the short exact sequence
 \begin{equation}\label{seseq2}
 1 \rightarrow L^\times \rightarrow J_L \rightarrow C_L \rightarrow 1,
 \end{equation} 
 we have the following commutative diagram with exact rows coming from the long exact sequences in cohomology corresponding to the short exact sequences (\ref{seseq1}) and (\ref{seseq2}).
\begin{equation}\label{E:CD1}
\begin{CD}
\hat{H}^{-1}(G,J_F) @>\alpha_F>> \hat{H}^{-1}(G,C_F) @>>>\hat{H}^{0}(G,F^\times) @>\kappa_F>>\hat{H}^{0}(G,J_F)  \\
 @VV\mbox{Def}^G_{H}V @VV\mbox{Def}^G_{H}V @VV\mbox{Def}^G_{H}V @VV\mbox{Def}^G_{H}V \\
 \hat{H}^{-1}(H,J_{L}) @>\alpha_L>> \hat{H}^{-1}(H,C_{L}) @>>> \hat{H}^{0}(H,L^\times) @>\kappa_{L}>> \hat{H}^{0}(H,J_{L})
\end{CD}.
\end{equation}
Recall that
\[ \Sha(L/K) =\frac{K^\times \cap N_{L/K}(J_L)}{N_{L/K}(L^\times)}=\mathrm{Ker}\left( \frac{K^\times}{N_{L/K}(L^\times)}\rightarrow \frac{J_K}{N_{L/K}(J_L)}\right) = \mathrm{Ker}(\kappa_L).  \]
Similarly, $\Sha(F/K) = \mathrm{Ker}(\kappa_F)$, so we can identify $h$ with the deflation map
\[\mathrm{Def}^G_H \colon \mathrm{Ker}(\kappa_F) \rightarrow \mathrm{Ker}(\kappa_L).  \]
Using (\ref{E:CD1}) we identify this with the map 
\[\cD \colon \mathrm{Coker}(\alpha_F) \rightarrow \mathrm{Coker}(\alpha_L) \]
induced by 
\[\mathrm{Def}^G_H \colon \hat{H}^{-1}(G,C_F) \rightarrow \hat{H}^{-1}(H,C_{L}).\]
Next, we apply the isomorphisms 
\[\Phi_G \colon \hat{H}^{-1}(G,C_F) \simeq \hat{H}^{-3}(G,\bZ) = H_2(G,\bZ) \] 
and 
\[\Psi_G\colon \hat{H}^{-1}(G,J_F) \simeq  \bigoplus_v \hat{H}^{-1}(G^v,F_v^\times)  \simeq \bigoplus_v \hat{H}^{-3}(G^v,\bZ)= \bigoplus_v H_2(G^v,\bZ) \] 

(cf. \cite[Chapter 7]{Cass}), together with the the corresponding isomorphisms for $H$, to the groups in the left half of (\ref{E:CD1}).  The discussion in \cite[p.\,198]{Cass} tells us that the diagrams

\[
\begin{CD}
\hat{H}^{-1}(G,J_L) @>\alpha_L>> \hat{H}^{-1}(G,C_L)\\
@V\Psi_GVV @V\Phi_GVV \\
\bigoplus_v H_2(G^v,\bZ) @>\gamma_L>> H_2(G,\bZ)
\end{CD}
\mbox{ }\mbox{ }
\mbox{ }
\mbox{ and  }
\mbox{ }
\mbox{ }
\mbox{ }
\begin{CD}
\hat{H}^{-1}(H,J_F) @>\alpha_F>> \hat{H}^{-1}(H,C_F)\\
@V\Psi_HVV @V\Phi_HVV \\
\bigoplus_v H_2(H^v,\bZ) @>\gamma_F>> H_2(H,\bZ) 
\end{CD}
\]
commute, and Theorem 1 in \cite{Horie} tells us that 
\[
\begin{CD}
\hat{H}^{-1}(G,C_F) @>>\Phi_G>H_2(G,\bZ)  \\
 @VV\mbox{Def}^G_{H}V @V\mbox{Rsd}^G_{H}VV \\
 \hat{H}^{-1}(H,C_{L}) @>\Phi_H>> H_2(H,\bZ) 
\end{CD}
\]
commutes, so these isomorphisms transform 
\[
\begin{CD}
\hat{H}^{-1}(G,J_F) @>\alpha_F>> \hat{H}^{-1}(G,C_F)   \\
 @. @VV\mbox{Def}^G_{H}V  \\
 \hat{H}^{-1}(H,J_{L}) @>\alpha_L>> \hat{H}^{-1}(H,C_{L}) 
\end{CD}
\mbox{ }\mbox{ }
\mbox{ }
\mbox{ into  }
\mbox{ }
\mbox{ }
\mbox{ }
\begin{CD}
\bigoplus_v H_2(G^v,\bZ) @>\gamma_F>> H_2(G,\bZ) \\
@.  @V\mathrm{Rsd}^G_{H}VV \\
\bigoplus_v H_2(H^v,\bZ) @>\gamma_L>> H_2(H,\bZ)
\end{CD}
\]
Thus, we can identify $\cD$ with $\cR_{L/F}$.  Finally, Corollary \ref{compat2} tells us that (\ref {H2Diag1}) is commutative. \hfill $\Box$

Next we give modified versions of Lemma \ref{ShaLem} and Corollary \ref{sha1cor} which are more useful for computations by using two different descriptions of the \textit{Schur multiplier} $H_2(-,\bZ)$.

Let $G$ be a finite group.  Following \cite[I.3]{Beyl}, we define $M(G)$ to be the group given by the Schur-Hopf formula,
\[M(G) := \frac{R \cap [F,F]}{[R,F]} ,\]
where 
\[1 \rightarrow R \rightarrow F \rightarrow G \rightarrow 1 \]
is any free presentation of $G$.  As discussed in \textit{loc. cit.}, the isomorphism class of $M(G)$ is independent of the choice of free presentation, and $M(-)$ becomes a functor from groups to abelian groups once we choose a fixed free presentation for each group $G$.  
\begin{lemma}\label{SchNat} $M(-)$ is naturally isomorphic to $H_2(-,\bZ)$.  That is, for every homomorphism of finite groups $\varphi \colon G \rightarrow H$ there is a commutative diagram of the form
\[
\begin{CD}
H_2(G,\bZ) @>\varphi_*>>H_2(H,\bZ) \\
@V\simeq  VV @V\simeq  VV \\
M(G) @> M(\varphi)>> M(H)
\end{CD}
.\]
\end{lemma}
\begin{proof} This follows from Proposition 5.5 of \cite[p.\,51]{Beyl}.   
\end{proof}
\begin{lemma}\label{ShaLem2} The diagram
\begin{equation}\label{H2Diag1V2}
\begin{CD}
\bigoplus_v M(G^v) @> \sum_v M(\iota_G^v)>> M(G) \\
@V \left( M(\pi^v) \right) VV @VM(\pi)VV \\
\bigoplus_v M(H^v) @> \sum_v M(\iota_H^v)>> M(H)
\end{CD}\end{equation}
commutes.  Let
\[\delta_F:= \sum_v M(\iota_G^v) \mbox{ }\mbox{ }\mbox{ }\mbox{ and  }\mbox{ }\mbox{ }\mbox{ }
\delta_L:= \sum_v M(\iota_H^v).\]
 Then $h$ can be identified with the map
\[\cS_{F/L}\colon \mathrm{Coker}(\delta_F)\rightarrow \mathrm{Coker}(\delta_L) \]  
induced by $M(\pi)$. 
\end{lemma}  
\begin{proof} We obtain (\ref{H2Diag1V2}) by applying Lemmas \ref{generalind} and \ref{SchNat} to (\ref{H2Diag1}).  This proves the commutivity of (\ref{H2Diag1V2}) and allows us to identify $\cR_{F/L}$ with $\cS_{F/L}$.
\end{proof}
\begin{corollary}\label{sha2cor} $g$ can be identified with the map
\[\cS_{L/L_1} \times \cS_{L/L_2} \colon \mathrm{Coker}(\delta_{L}) \rightarrow \mathrm{Coker}(\delta_{L_1}) \times \mathrm{Coker}(\delta_{L_2}). \]
\end{corollary}
The \textit{exterior square} of a finite abelian group $G$ is defined as
\[G \wedge G := G \otimes G / \langle g \otimes g | g \in G  \rangle .\]  

\begin{lemma}\label{SchNat2} If $G$ is abelian, then $M(G)$ is naturally isomorphic to $G \wedge G$.  That is, for every homomorphism of finite abelian groups $\varphi\colon G \rightarrow H$ there is a commutative diagram of the form
\[
\begin{CD}
M(G) @> M(\varphi)>> M(H) \\
@V\simeq  VV @V\simeq  VV \\
G \wedge G @> \varphi \wedge \varphi>> H \wedge H
\end{CD}
\]
where $\varphi \wedge \varphi$ is the map induced by $\varphi \otimes \varphi$.
\end{lemma}
\begin{proof} This follows from 4.5 and 4.7 in \cite[I.4]{Beyl}.
\end{proof}

\begin{lemma}\label{ShaLem3} The diagram
\begin{equation}\label{H2Diag1V3}
\begin{CD}
\bigoplus_v G^v \wedge G^v @> \sum_v \iota_G^v \wedge \iota_G^v>> G \wedge G \\
@V \left( \pi^v \wedge \pi^v \right) VV @V \pi \wedge \pi VV \\
\bigoplus_v H^v \wedge H^v @> \sum_v \iota_H^v \wedge \iota_H^v>> H \wedge H
\end{CD}\end{equation}
commutes.  Let
\[\epsilon_F:= \sum_v \iota_G^v \wedge \iota_G^v \mbox{ }\mbox{ }\mbox{ }\mbox{ and  }\mbox{ }\mbox{ }\mbox{ }
\epsilon_L:= \sum_v \iota_H^v \wedge \iota_H^v.\]
 Then $h$ can be identified with the map
\[\cT_{F/L}\colon \mathrm{Coker}(\epsilon_F)\rightarrow \mathrm{Coker}(\epsilon_L) \]  
induced by $\pi \wedge \pi$.   
\end{lemma}  
\begin{proof} We obtain (\ref{H2Diag1V3}) by applying Lemma \ref{SchNat2} to (\ref{H2Diag1V2}).  This proves the commutivity of (\ref{H2Diag1V3}) and allows us to identify $\cS_{F/L}$ with $\cT_{F/L}$. 
\end{proof}
\begin{corollary}\label{sha3cor} $g$ can be identified with the map
\[\cT_{L/L_1} \times \cT_{L/L_2} \colon \mathrm{Coker}(\epsilon_{L}) \rightarrow \mathrm{Coker}(\epsilon_{L_1}) \times \mathrm{Coker}(\epsilon_{L_2}). \]
\end{corollary}

\section{The Second Obstruction}
We begin this section by constructing an exact sequence that contains the map \[f \colon \Sha(L_1/K) \times \Sha(L_2/K) \rightarrow \Sha(L_1,L_2/K) \]
defined by
\[\left(x  N(L_1^\times) ,yN(L_2^\times)    \right) \mapsto xy^{-1}N(L_1^\times)N(L_2^\times).\] 

\begin{proposition}\label{SixTerm} If $L_1,L_2$ is a pair of finite extensions of $K$, then there is an exact sequence of the form
\vskip 2mm
\begin{equation}\label{Six} 1 \rightarrow \frac{K^\times \cap N(J_{L_1})\cap N(J_{L_2})}{N(L_1^\times) \cap N(L_2^\times)} \rightarrow \Sha(L_1/K) \times \Sha(L_2/K)  \stackrel{f}{\rightarrow}  \Sha(L_1,L_2/K) \rightarrow
\end{equation}
\[\rightarrow \frac{J_K}{K^\times \left(N(J_{L_1})\cap  N(J_{L_2}) \right)} \rightarrow \frac{J_K}{K^\times N(J_{L_1})} \times \frac{J_K}{K^\times N(J_{L_2})} \rightarrow \frac{J_K}{K^\times N(J_{L_1})N(J_{L_2})} \rightarrow 1. \]
\vskip 2mm
\end{proposition}
The proof of Proposition \ref{SixTerm} uses the following elementary lemma.
\begin{lemma}\label{ShortGrp}
Let $\mathcal{A}$ be an abelian group with subgroups $\mathcal{B}$
and $\mathcal{C}.$ The sequence
$$1 \longrightarrow
\frac{\mathcal{A}}{\mathcal{A}\cap \mathcal{B}} \stackrel{\varphi}{\longrightarrow}
\frac{\mathcal{A}}{\mathcal{B}} \times
\frac{\mathcal{A}}{\mathcal{C}} \stackrel{\psi}{\longrightarrow}
\frac{\mathcal{A}}{\mathcal{B}\mathcal{C}} \longrightarrow 1,
$$
where $\varphi$ and $\psi$ are defined by
$$
\varphi(x\mathcal{A}\cap\mathcal{B}) = (x\mathcal{B} , x\mathcal{C}) \ \ \text{and} \ \
\psi(x\mathcal{B} , y\mathcal{C}) = xy^{-1}\mathcal{B}\mathcal{C},
$$
is exact.
\end{lemma}

\noindent {\it Proof of Lemma \ref{ShortGrp}.} We obtain the commutative diagram
\vskip 2mm
\[ 
\begin{CD}
1 @>>> \frac{ K^\times}{ N(L_1^\times)\cap  N(L_2^\times)} @>>> \frac{K^\times}{ N(L_1^\times)} \times \frac{K^\times}{ N(L_2^\times)} @>>> \frac{K^\times}{ N(L_1^\times)N(L_2^\times)} @>>>1  \\
@. @VVV @VVV @VVV @. \\
1 @>>> \frac{ J_K}{N(J_{L_1})\cap  N(J_{L_2})} @>>> \frac{J_K}{ N(J_{L_1})} \times \frac{J_K}{ N(J_{L_2})} @>>> \frac{J_K}{ N(J_{L_1})N(J_{L_2})} @>>>1
\end{CD}
\]
\vskip 3mm
by applying Lemma \ref{ShortGrp} twice.  The top row corresponds to $\mathcal{A}=K^\times$, $\mathcal{B}=N(L_1^\times)$, and $\mathcal{C}=N(L_2^\times)$, while the bottom row corresponds to $\mathcal{A}=J_K$, $\mathcal{B}=N(J_{L_1})$, and $\mathcal{C}=N(J_{L_2})$.  The vertical maps are induced by the inclusion $K^\times \rightarrow J_K$.  Applying the snake lemma to this diagram gives (\ref{Six}).\hfill $\Box$

It follows from the exactness of (\ref{Six}) that 
\begin{equation}\label{Sha2Rew}
\Sha_2(L_1,L_2/K) \simeq  \frac{K^\times N(J_{L_1})\cap K^\times N(J_{L_2})}{K^\times( N(J_{L_1})\cap  N(J_{L_2}))}. 
\end{equation}

We give an upper bound for the order of $\Sha_2(L_1,L_2/K)$ by using class field theory to estimate the order of this quotient.  For a global field $F$, we let $C_F$ denote the idele class group of $F$.

\begin{lemma}\label{Sha2est} Let $L_1,L_2$ be a pair of finite extensions of $K$, and let $L=L_1L_2.$  If $M_i$ is the maximal abelian subextension of $L_i/K$, and $M$ is the maximal abelian subextension of $L/K$, then
\[|\Sha_2(L_1,L_2/K)| \leq \frac{[M:K]}{[M_1M_2:K]}. \]
In particular, if $M=M_1M_2$, then the second obstruction is trivial.  
\end{lemma} 
\begin{proof} Clearly $N(J_{L}) \leq N(J_{L_1}) \cap N(J_{L_2})$, so it follows from (\ref{Sha2Rew}) that $\Sha_2(L_1,L_2/K)$ is isomorphic to a quotient of 
\begin{equation}\label{Sha2Rew2}
\frac{K^\times N(J_{L_1})\cap K^\times N(J_{L_2})}{K^\times N(J_{L}) } \simeq \frac{N(C_{L_1}) \cap N(C_{L_2})}{N(C_L)}.
\end{equation}
According to \cite[Exercise 8]{Cass},
\[N(C_L)=N(C_M), \mbox{ } \mbox{ } \mbox{ } \mbox{ } \mbox{ } \mbox{ } \mbox{ }N(C_{L_i})=N(C_{M_i}),\]
by \cite[p.\,55]{ArT},
\[N(C_{M_1}) \cap N(C_{M_2})= N(C_{M_1M_2}),\]
and by \cite[p.\,172 Theorem 5.1 B]{Cass},
\[[C_K:N(C_M)]=[M:K] \mbox{ } \mbox{ } \mbox{ } \mbox{ and } \mbox{ } \mbox{ } \mbox{ }[C_K:N(C_{M_1M_2})]=[M_1M_2:K], \] 
so 
\[|\Sha_2(L_1,L_2/K) | \leq \left| \frac{N(C_{L_1}) \cap N(C_{L_2})}{N(C_L)} \right| = \left| \frac{N(C_{M_1M_2})}{N(C_M)}   \right| = \frac{[M:K]}{[M_1M_2:K]}. \]
\end{proof}
Another approach to the second obstruction is to consider the map
\[\varphi \colon  J_{L_1}/L^{\times}_1 N_{L/L_1}(J_L) \times
J_{L_2}/L^{\times}_2 N_{L/L_2}(J_L) \longrightarrow J_{K}/K^{\times}
N_{L/K}(J_L) \]
induced by the product of norm maps $N_{L_1/K}$ and $N_{L_2/K}$ as in \cite{PlR} and \cite{PoR}.  
\begin{lemma}\label{L:2}If $\varphi$ is injective, then
\begin{equation}\label{KeyIdentity}
K^\times \cap N(J_{L_1})N(J_{L_2}) = (K^\times \cap N(J_{L}))N(L_1^\times)N(L_2^\times) .
\end{equation}
In particular, $\Sha_2(L_1,L_2/K)=\{ 1\}$.  \end{lemma}

\begin{proof}
See the proof of Lemma 3 in \cite{PoR}.
\end{proof}  

\begin{corollary}\label{ShaSurjCor} If $L_1 \subset L_2$ and $\varphi$ is injective, then the natural map 
 \[\Sha(L_2/K) \rightarrow \Sha(L_1/K) \]
 is surjective.
\end{corollary}
\begin{proof} In this case, (\ref{KeyIdentity}) takes the form
\[ K^\times \cap N(J_{L_1}) = (K^\times \cap N(J_{L_2}))N(L_1^\times) .\]
\end{proof}

Lemma 3 of \cite{PoR} says that $\varphi$ is a bijection if $L_1,L_2$ is a pair of linearly disjoint Galois extensions of $K$.  This result can be strengthened as follows.  
\begin{lemma}\label{L:3B}Let $L_1,L_2$ be a pair of Galois extensions of $K$, let $L=L_1L_2$ and $E=L_1 \cap L_2$, and let $(-)^{ab}$ denote the abelianization functor.  $\varphi$ is injective iff the natural map
$\Gal(L/E)^{ab} \rightarrow \Gal(L/K)^{ab} $
is injective.  In particular, $\varphi$ is injective whenever $L_1$ and $L_2$ are both abelian extensions of $K$.    
\end{lemma}
\begin{proof} We can factor $\varphi$ as 
\[ J_{L_1}/L^{\times}_1 N_{L/L_1}(J_L) \times
J_{L_2}/L^{\times}_2 N_{L/L_2}(J_L) \stackrel{\varphi_0}{\longrightarrow}  J_{E}/E^{\times}
N_{L/E}(J_L) \stackrel{N_{E/K}}{\longrightarrow} J_{K}/K^{\times}N_{L/K}(J_L),\]
where $\varphi_0$ is induced by the product of the norm maps $N_{L_1/E}$ and $N_{L_2/E}.$  $L_1,L_2$ is a pair of linearly disjoint Galois extensions of $E$, so $\varphi_0$ is an isomorphism.  It follows that $\varphi$ is injective iff $N_{E/K}$ is injective.  Since the Tate isomorphisms commute with corestriction \cite[p.\,197]{Cass}, there is a commutative diagram 
\[
\begin{CD}
J_{E}/E^{\times}N_{L/E}(J_L) @>N_{E/K}>>  J_K/K^{\times} N_{L/K}(J_L) \\
@VV\simeq V   @VV\simeq V\\
\Gal(L/E)^{ab} @>>> \Gal(L/K)^{ab}
\end{CD}
\]
and the claim follows.
\end{proof}
\section{Proof of Theorem 1}
Let $L_1,L_2$ be a pair of finite abelian extensions of $K$, let $L=L_1L_2$, and let $E=L_1 \cap L_2$.  We begin the proof of Theorem 1 by computing $\mathrm{Coker}(g)$.  
 
\begin{lemma}\label{LastBlow}
\[\mathrm{Coker}(g) \simeq \Sha(E/K)  .\]
\end{lemma}   
Let $G$, $G_i$, and $G_E$ denote the Galois groups $\Gal(L/K)$, $\Gal(L_i/K)$, and $\Gal(E/K)$ respectively, and let $G^v$, $G_i^v$, and $G_E^v$ denote the decomposition groups of (compatible) fixed extensions of $v$ to the fields $L$, $L_i$, and $E$.  Let $\pi_i\colon G \rightarrow G_{i}$, $\rho\colon G \rightarrow G_E$, $\rho_i\colon G_i \rightarrow G_E$, and $\rho_i^v \colon G_{i}^v \rightarrow G_E^v$ denote the canonical projection maps, and let $\iota^v \colon G^v \rightarrow G$, $\iota_i^v \colon G_{i}^v \rightarrow G_{i}$, and $\iota_E^v \colon G_E^v \rightarrow G_E$ denote the canonical inclusion maps.

By Lemma \ref{ShaLem3} and Corollary \ref{sha3cor} there are commutative diagrams
\begin{equation}\label{triplediag}
\begin{CD}
\bigoplus_v G^v \wedge G^v @>\epsilon_L>> G \wedge G \\
@VVV @V \pi_i \wedge \pi_i VV \\
\bigoplus_v G_i^v \wedge G_i^v @>\epsilon_{L_i}>> G_i \wedge G_i \\
@V\cL_iVV @V \rho_i \wedge \rho_i VV \\
\bigoplus_v G_E^v \wedge G_E^v @>\epsilon_E>> G_E \wedge G_E
\end{CD}
\end{equation}
for $i=1,2$ where 
\[\epsilon_L = \sum_v \iota^v \wedge \iota^v  ,\mbox{ }\mbox{ }\mbox{ }\mbox{ }\mbox{ }\mbox{ }\mbox{ } \epsilon_{L_i} = \sum_v \iota_{i}^v \wedge \iota_{i}^v , \] 
\[ \epsilon_E = \sum_v \iota_E^v \wedge \iota_E^v ,\mbox{ }\mbox{ }\mbox{ and }\mbox{ }\mbox{ }\mbox{ }\mbox{ } \cL_i = \left( \rho_i^v \wedge \rho_i^v \right),\]
and we can identify $g$ with the map
\[\cT \colon \mathrm{Coker}(\epsilon_L) \rightarrow \mathrm{Coker}(\epsilon_{L_1}) \times \mathrm{Coker}(\epsilon_{L_2}) \]
induced by the map
\[\cT_0 \colon G \wedge G \rightarrow (G_{1} \wedge G_{1}) \times (G_{2} \wedge G_{2})\]
defined by 
\[\cT_0( a \wedge b) = \bigg( \pi_1(a) \wedge \pi_1(b) , \pi_2(a) \wedge \pi_2(b) \bigg)    \]
for $a,b \in G$.  

We start by analyzing $\cT_0$.  Given subsets $A$ and $B$ of an abelian group $C$, we let $A \wedge B$ denote the set of all sums in $C \wedge C$ of the form $\sum_i a_i \wedge b_i$ with $a_i \in A$ and $b_i \in B$.  Let $\mu$ be a fixed section of $\rho$ and define a section $\mu_1$ of $\rho_1$ by $\mu_1 = \pi_1 \circ \mu$.       
\begin{lemma}\label{wedgerel}
\[  \bigg(\mu_1(G_E) \wedge \mathrm{Ker}(\rho_1), 0 \bigg) = \cT_0
\bigg( \mu(G_E) \wedge \mathrm{Gal}(L/L_2) \bigg) , \]
\[
 \bigg(\mathrm{Ker}(\rho_1)\wedge \mu_1(G_E), 0 \bigg) = \cT_0 \bigg(\mathrm{Gal}(L/L_2) \wedge \mu(G_E) \bigg),
\]
and
\[
 \bigg( \mathrm{Ker}(\rho_1) \wedge \mathrm{Ker}(\rho_1), 0  \bigg) = \cT_0 \bigg( \mathrm{Gal}(L/L_2)\wedge \mathrm{Gal}(L/L_2) \bigg)  . 
\]
\end{lemma}
\begin{proof} Direct computation.  We note that $\pi_1 (\Gal(L/L_2)) = \Gal(L_1/E) = \mathrm{Ker}( \rho _1),$ $\pi_2 (\Gal(L/L_2))= 1$, $\pi_1 (\mu(G_E))= \mu_1(G_E),$ and that $g_1 \wedge g_2 = 0$ whenever $g_2=1$. 
\end{proof}
These identities allow us to compute the cokernel of $\cT_0$.
\begin{lemma}\label{2ndLastBlow}
\[\mathrm{Coker}(\cT_0) \simeq G_E \wedge G_E. \]
\end{lemma}
\begin{proof}
Since $\mathrm{Coker}(\cT_0)$ and $G_E \wedge G_E$ are finite groups, it suffices to construct a surjective homomorphism from each to the other.  We define a surjective homomorphism 
\[P_0 \colon (G_{1} \wedge G_{1}) \times (G_{2} \wedge G_{2}) \rightarrow \ G_E \wedge G_E \]
by
\[P_0(a\wedge b , 0) = \rho_1(a) \wedge \rho_1(b) \mbox{ }\mbox{ }\mbox{ }\mbox{ and } \mbox{ }\mbox{ }\mbox{ } P_0(0,c \wedge d) = -\rho_2(c) \wedge \rho_2(d) \]
for $a,b \in G_{1}$ and $c,d \in G_{2}$.   $P_0 \circ \cT_0 =0$, so $P_0$ induces a surjective homomorphism 
\[P \colon \mathrm{Coker}(\cT_0) \rightarrow G_E \wedge G_E.\]  

To get a homomorphism in the other direction, we first define a set map
\[S_0 \colon G_E \times G_E \rightarrow \mathrm{Coker}(\cT_0) \]
by 
\[S_0(e,f) = \bigg(\mu_1(e) \wedge \mu_1(f),0 \bigg) + \mathrm{Im}(\cT_0) \]
for $e,f \in G_E$.  If $e_1,e_2 \in G_E$, then $\mu_1(e_1e_2)\mu_1(e_1)^{-1}\mu_1(e_2)^{-1} \in \mathrm{Ker}(\rho_1)$, so it follows from Lemma \ref{wedgerel} that 
\[S_0(e_1e_2,f)-S_0(e_1,f)-S_0(e_2,f)=\bigg(\mu_1(e_1e_2)\mu_1(e_1)^{-1}\mu_1(e_2)^{-1} \wedge \mu_1(f), 0 \bigg)+\mathrm{Im}(\cT_0) = 0. \] 
A similar calculation can be done for the second argument, so $S_0$ is bilinear and induces a homomorphism $S \colon G_E \wedge G_E \rightarrow \mathrm{Coker}(\cT_0)$.  It remains to show that $S$ is surjective.  Since $\pi_2 \wedge \pi_2$ is surjective, we have 
\[\mathrm{Coker}(\cT_0) = \bigg(G_1 \wedge G_1, 0 \bigg) +\mathrm{Im}(\cT_0), \]
and it suffices to show that  
\[ \bigg(g_1 \wedge g_2,0\bigg)+\mathrm{Im}(\cT_0) =\bigg(( \mu_1 \circ \rho_1)(g_1) \wedge ( \mu_1 \circ \rho_1)(g_2),0\bigg) +\mathrm{Im}(\cT_0) \]
for all $g_1,g_2 \in G_1$.  Since
\[ ( \mu_1 \circ \rho_1)(g_i)g_i^{-1} \in \mathrm{Ker}(\rho_1),\]
for $i=1,2$, this follows from Lemma \ref{wedgerel}. 
\end{proof}
\noindent {\it Proof of Lemma \ref{LastBlow}.} Since $g$ can be identified with $\cT$ and $\Sha(E/K)$ can be identified with $\mathrm{Coker}(\epsilon_E)$, it suffices to prove that $\mathrm{Coker}(\cT) \simeq \mathrm{Coker}(\epsilon_E)$.  $P$ induces a homomorphism
\[ \cP \colon  \mathrm{Coker}(\cT) \rightarrow \mathrm{Coker}(\epsilon_E) \]
which must be surjective since $P$ is surjective.  Since the maps $\cL_i$ in (\ref{triplediag}) are surjective, and since $P$ is injective, a short diagram chase shows that $\cP$ must also be injective.\hfill $\Box$
\vskip 1mm
\noindent {\it Proof of Theorem 1.} By Lemma \ref{Sha2est}, $\Sha_2(L_1,L_2/K)=\{1\}$, so $\Sha(L_1,L_2/K)=\Sha_1(L_1,L_2/K)$.  According to Lemmas \ref{FirstRed} and \ref{LastBlow}, $\Sha(L_1,L_2/K)$ is isomorphic to a quotient of $\Sha(E/K)$.  Since both of these groups are finite, it suffices to show that $\Sha(E/K)$ is isomorphic to a quotient of $\Sha(L_1,L_2/K)$.  Consider the map 
\[j \colon \Sha(L/K) \rightarrow \Sha(L,E/K) =\Sha(E/K)  \]
defined by 
\[xN(L^\times) \mapsto xN(E^\times). \]
$L,E$ is a pair of abelian extensions of $K$, so Lemma \ref{L:3B} and Corollary \ref{ShaSurjCor} guarantee that $j$ is surjective.  $j$ factors through $\Sha(L_1,L_2/K)$, so $\Sha(E/K)$ is a homomorphic image of $\Sha(L_1,L_2/K)$.   \hfill $\Box$

\section{Examples and Discussion} 
In this section we describe several applications of the methods developed in the previous sections and discuss some problems related to the multinorm principle. 
\vskip 1mm
\noindent \newline \textbf{Example 1.} If $L_1, L_2$ is a linearly disjoint pair of finite Galois extensions of a global field $K$, then we can recover the main theorem of \cite{PoR} by using the results from $\S2$ and $\S3$ to prove that $\Sha(L_1,L_2/K)= \{1\}.$  To prove that $\Sha_1(L_1,L_2/K)= \{1 \}$, it suffices to show that
\[g \colon  \Sha(L/K) \rightarrow \Sha(L_1/K) \times \Sha(L_2/K)   \]     
is surjective.  Let $G=\Gal(L/K)$, $G_i=\Gal(L_i/K)$, and let $\pi_i \colon G \rightarrow G_i$ be the natural projection map.  By Corollary \ref{sha2cor} it suffices to show that 
\begin{equation}\label{Msurj}
M(\pi_1) \times M(\pi_2)\colon  M(G) \rightarrow M(G_1) \times M(G_2) 
\end{equation}
is surjective.  Let $\iota_i \colon G_i \rightarrow G$ be the monomorphism corresponding to the natural identification 
\[ G_i = \Gal(L_i/K) \simeq \Gal(L/L_{3-i}) \leq G.\]
  Then $\pi_i \circ \iota_i = \mathrm{id}_{G_i}$ and $\pi_{3-i} \circ \iota_i$ is the trivial homomorphism, so $ M(\pi_i) \circ M(\iota_i) = \mathrm{id}_{M(G_i)}$ while $M(\pi_{3-i}) \circ M(\iota_i)$ is the zero map, and the surjectivity of (\ref{Msurj}) follows.
  
If $M_i$ and $M$ denote the maximal abelian subextensions of $L_i/K$ and $L_1L_2/K$ respectively, then $M=M_1M_2$, so $\Sha_2(L_1,L_2/K) = \{1 \}$ by Lemma \ref{Sha2est}.

\vskip 1mm
\noindent \newline \textbf{Example 2.} If $L_1, L_2$ is a pair of finite extensions of a global field $K$, let us say that $L_1,L_2$ satisfies the \textit{intersection principle} if 
\[K^\times \cap N(J_{L_1}) \cap N(J_{L_2}) = N(L_1^\times) \cap N(L^\times). \]
The obstruction to this local-global principle is given by
\[ \Sha_\cap(L_1,L_2/K) := \frac{K^\times \cap N(J_{L_1}) \cap N(J_{L_2}) } { N(L_1^\times) \cap N(L^\times)}.\]
This group naturally arises as the first nontrivial term in (\ref{Six}), and we can truncate (\ref{Six}) to obtain the short exact sequence
\begin{equation}\label{int1}
1 \rightarrow \Sha_\cap(L_1,L_2/K) \rightarrow \Sha(L_1/K) \times \Sha(L_2/K) \rightarrow 
\Sha_1(L_1,L_2/K) \rightarrow 1 . 
\end{equation}
One possibility this suggests is that we may be able to learn about the first obstruction indirectly by studying the intersection problem.  On the other hand, we can use (\ref{int1}) to determine if the intersection principle holds whenever we understand $\Sha_1(L_1,L_2/K)$.  In particular, we have the following corollaries to the main theorem of \cite{PoR} and Theorem 1 of this paper.
\begin{corollary} If $L_1,L_2$ is a pair of finite separable extensions of $K$ with linearly disjoint Galois closures, then
\[\Sha_\cap(L_1,L_2/K) \simeq \Sha(L_1/K) \times \Sha(L_2/K). \]
\end{corollary}
 \begin{corollary} If $L_1,L_2$ is a pair of finite abelian extensions of $K$, then
 \[ |\Sha_\cap(L_1,L_2/K)| = \frac{|\Sha(L_1/K)||\Sha(L_2/K)|}{|\Sha(L_1\cap L_2/K)|}. \]
\end{corollary}

\vskip 1mm
\noindent \newline \textbf{Example 3.} The map $\varphi$ defined in $\S3$ may fail to be injective even if $E=L_1 \cap L_2$ is a cyclic extension of $K$.  Let $K=\bQ$, $L_1= \bQ(i,2^{1/4})$, and $L_2= \bQ(\sqrt{2},\sqrt{3})$.  Then $L=\bQ(i,2^{1/4},\sqrt{3})$ and $E=\bQ(\sqrt{2})$.  Let $G=\mathrm{Gal}(L/K)$, let $H=\mathrm{Gal}(L/E)$, let $\tau \in H$ be the automorphism defined by complex conjugation, and let $\sigma \in G$ be the automorphism which sends $2^{1/4}$ to $i2^{1/4}$ and fixes $i$ and $\sqrt{3}$.  Then $H=H^{ab}$ and $[\sigma, \tau]$ is a non-trivial element of $\mathrm{Ker}(H^{ab} \rightarrow G^{ab})$, so it follows from Lemma \ref{L:3B} that $\varphi$ is not injective.

\bibliographystyle{amsplain}

\end{document}